\documentclass[a4paper, 12pt]{article}

\usepackage[english]{babel}
\usepackage{amssymb}
\usepackage{amsthm}
\usepackage{amsmath}
\usepackage[latin1]{inputenc}
\usepackage[T1]{fontenc}
\usepackage{lmodern}
\usepackage[backend=biber, style=numeric, isbn=false, url=false, doi=false, eprint=false, sorting=nyvt, maxnames=10, giveninits=true]{biblatex}
\usepackage[babel,german=guillemets]{csquotes}
\usepackage{caption}
\usepackage{float}
\usepackage{scrhack}
\usepackage{pdfpages}

\usepackage{geometry}
\newtheorem{lemma}{Lemma}[section]
\newtheorem{theorem}[lemma]{Theorem}
\newtheorem{korollar}[lemma]{Corollary}
\newtheorem{anmerkung}[lemma]{Remark}
\newtheorem{definition}[lemma]{Definition}

\numberwithin{equation}{section}

\newcommand*{\N}{\mathbb{N}}
 
\newcommand*{\e}{\operatorname{e}}

\renewbibmacro{in:}{%
  \ifboolexpr{%
   test {\ifentrytype{article}}%
   }{}{\printtext{\bibstring{in}\intitlepunct}}%
}

\DeclareNameAlias{default}{last-first}

\begin{document}

\title{On convergence of discrete methods of least squares on equidistant nodes}
\author{Ren\'e Goertz}
\date{\today}
\maketitle

\begin{section}*{Abstract}
\label{abstract}

We consider the well-known method of least squares on an equidistant grid with $N+1$ nodes on the interval $[-1,1]$ with the goal to approximate a function $f\in\mathcal{C}\left[-1,1\right]$ by a polynomial of degree $n$. We investigate the following problem: For which\linebreak ratio $N/n$ and which functions do we have uniform convergence of the least square operator ${LS}_n^N:\mathcal{C}\left[-1,1\right]\rightarrow\mathcal{P}_n$? We investigate this problem with a discrete\linebreak weighting of the Jacobi-type. Thereby we describe the least square operator ${LS}_n^N$ by the expansion of a function by Hahn polynomials $Q_k\left(\cdot;\alpha,\beta,N\right)$. Without additional assumptions to functions $f\in\mathcal{C}\left[-1,1\right]$ it can not be guaranteed uniform convergence. But with $\alpha=\beta$ and additional assumptions to $f$ and $\left(N_n\right)_{n\in\mathbb{N}}$ we obtain convergence and prove the following results: For an $\alpha\geq0$ let $f\in\left\{g\in\mathcal{C}^\infty\left[-1,1\right]:\ \lim\limits_{n\to\infty}{\sup\limits_{x\in[-1,1]}{\left\lvert g^{(n)}(x)\right\rvert}\frac{n^{\alpha+1/2}}{2^nn!}}=0\right\}$ and let $(N_n)_{n}$ be a sequence of natural numbers with $N_n\geq2n(n+1)$. Then the method of least squares ${LS}_n^{N_n}[f]$ converges uniform on $[-1,1]$. Before we determine the maximum error (\glqq worst case\grqq) with respect to the sup norm on the classes\\ $\mathcal{K}_{n+1}:=\left\{f\in\mathcal{C}^{n+1}\left[-1,1\right]:\ \sup\limits_{x\in[-1,1]}{\left\lvert f^{(n+1)}(x)\right\rvert\leq1}\right\}$.
\end{section}

\begin{section}{Introduction and statement of the main results}
\label{introduction}

It is over $200$ years ago since Legendre, Gau\ss\ and others started working with the method of least squares (cf., e.g., \cite{merriman1877history}). Since then, the method is used in many areas of mathematics and is nowadays a basic tool of applied mathematics (cf., e.g., \cite{abdulle2002200}, \cite{arens2015mathematik}, \cite{bjorck1996numerical}, \cite{gautschi2011numerical}, \cite{schwarz2011numerische}). Our focus in this paper is the pure approximation property of the method.
\par\medskip
The method of least squares is defined as follows (cf., e.g., \cite[p. 59]{gautschi2011numerical}, \cite[p. 217]{gautschi2004orthogonal}, \cite[p. 291]{milne1949numerical}):\\
Let $x_\mu\in\left[a,b\right]$ be distinct nodes for $\mu=0,\dots,N$. Further let $\omega:\left[a,b\right]\rightarrow\mathbb{R}$ be a weight-function, which is positive on $\left\{x_\mu\right\}_{\mu=0}^N$. For a $n\leq N$ let $U$ be a subspace of $\mathcal{C}\left[a,b\right]$ with $\mbox{dim}\, U=n+1$ on $\left\{x_\mu\right\}_{\mu=0}^N$. The least square operator ${LS}_n^N:\mathcal{C}\left[a,b\right]\rightarrow U$ is unique defined by
\begin{align*}
\sum\limits_{\mu=0}^N{\left({LS}_n^N[f]\left(x_\mu\right)-f\left(x_\mu\right)\right)^2 \omega\left(x_\mu\right)}=\min_{\varphi\in U}{\sum\limits_{\mu=0}^N{\left(\varphi\left(x_\mu\right)-f\left(x_\mu\right)\right)^2 \omega\left(x_\mu\right)}}.
\end{align*}
In this paper we investigate the standard case:
\begin{itemize}
	\item $U=\mathcal{P}_n$ is the space of polynomials of degree $n$,
	\item $\left\{x_0,\dots,x_N\right\}$ is an equidistant grid with $N+1$ nodes on the interval $[-1,1]$, i.\ e.\ $x_\mu=-1+2\mu/N$ for $\mu=0,\dots,N$.
\end{itemize}
This situation is often occur in the practice: Since centuries polynomials are an intensive investigated function class to approximation. Moreover they can be applied effective on computers, because only the elementary operations addition and multiplication will be used for every computation. Equidistant collected informations are often exist, especially due to the data collection on big (often multidimensional) equidistant grids.
\par\medskip
Without loss of generality let the interval $[a,b]$ be the standard interval $[-1,1]$ in this paper.
\par\medskip
We investigate the following \textbf{problem}:
\par\medskip
For which functions $f\in K\subset\mathcal{C}\left[-1,1\right]$ and which ratio $N/n$ converges the sequence $\left({LS}_n^N[f]\right)$ uniformly?
\par\medskip
To investigate the above problem we describe the least square operator ${LS}_n^N$ by the expansion of a function by Hahn polynomials $Q_k\left(\cdot;\alpha,\beta,N\right)$. The Hahn polynomials $Q_k\left(\cdot;\alpha,\beta,N\right)$ are classical discrete orthogonal polynomials on the interval $I=\left[0,N\right]$ of degree $k$. They are orthogonal on $I$ with respect to the inner product
\begin{align*}
\left<f,g\right>_\omega:=\sum\limits_{i=0}^N{f(i)g(i)\omega(i)},
\end{align*}
where $\omega$ is the weight-function given by
\begin{align*}
\omega(x):=\binom{\alpha+x}{x}\binom{\beta+N-x}{N-x}.
\end{align*}
They are normalized by
\begin{align*}%\label{orthoghahn1ae}
\left<Q_k(\cdot;\alpha,\beta,N),Q_k(\cdot;\alpha,\beta,N)\right>_\omega=\frac{(-1)^k(k+\alpha+\beta+1)_{N+1}(\beta+1)_k k!}{(2k+\alpha+\beta+1)(\alpha+1)_k(-N)_k N!}
\end{align*}
(cf., e.g., \cite[p. 204]{koekoek2010hypergeometric}).
\par\medskip
It is well-known that the least square operator ${LS}_n^N$ can be represented by use of Hahn polynomials (cf., e.g., \cite[p. 62-63]{gautschi2011numerical}, \cite[p. 270]{milne1949numerical}, \cite[p. 218-232]{wernerfunktionalanalysis}):
\begin{align}\label{darstellungmethodediskret}
{LS}_n^N[f]=\sum\limits_{k=0}^n{\frac{\left<f\left(\frac{2}{N}(\cdot)-1\right),Q_k\right>_\omega}{\left<Q_k,Q_k\right>_\omega}Q_k\left(\frac{N}{2}(1+\cdot)\right)},
\end{align}
where $f\in\mathcal{C}\left[-1,1\right]$.
\par\medskip
Without an additional assumption to the functions $f\in\mathcal{C}\left[-1,1\right]$ the uniform convergence of the sequence $\left({LS}_n^N[f]\right)$ can not be guaranteed (cf., e.g., \cite[p. 106, Satz 4.10]{schonhage1971approximationstheorie}). Hence we have to reduce the function class.
\par\medskip
The Hahn polynomials $Q_n\left(\cdot;\alpha,\beta,N\right)$ can be interpreted as a discretization of the Jacobi polynomials $P_n^{\alpha,\beta}$. Because for a fixed $n$ the following relation between Hahn polynomials $Q_n\left(\cdot;\alpha,\beta,N\right)$ and Jacobi polynomials $P_n^{\alpha,\beta}$ is well-known.
\begin{align*}%\label{grenzhahnjacobi}
\lim_{N\to\infty}{(-1)^n\binom{n+\alpha}{n}Q_n\left(\frac{N}{2}(1+x);\alpha,\beta,N\right)}=P_n^{\beta,\alpha}(x),
\end{align*}
for each $x\in[-1,1]$ (cf., e.g., \cite[p. 45]{nikiforov1991classical}).
\par\medskip
For all approximation results in this paper we consider the important symmetric (so-called ultraspherical) case $\alpha=\beta$. The close connection between the series expansion of a function by Jacobi polynomials and the series expansion by Hahn polynomials, cf. (\ref{darstellungmethodediskret}), which have been proved in \cite{goertzoeffner2016hahn}, is the motivation for my here presented, from Thomas Sonar and Tom Koornwinder inspired investigations: The series expansion by Jacobi polynomials is in the last decades a proved method to modelling. However it has to be evaluated integrals to calculate the coefficients. Usually this is done by discretization with the aid of methods of quadrature theory. Since it has to be discretized to approximate the integral, the question is obviously, if equivalently results can be obtained directly with the aid of discrete orthogonal polynomials, therefore without calculation of integrals. The main Theorem is:

\begin{theorem}\label{theoremgleichmaessigekonvergenz1e}
Let $\alpha>-\frac{1}{2}$ and let for $N\in\N$
\begin{align*}
n(\alpha,N):=\frac{1}{2}-\alpha+\frac{1}{2}\sqrt{(2\alpha+1)(2\alpha+2N+1)}.
\end{align*}
Further let
\begin{align*}
D_{n,N}:=\frac{2^{n+1}\Gamma\left(n+2\alpha+2\right)\Gamma\left(n+\alpha+2\right)}{(n+1)!\Gamma\left(2n+2\alpha+3\right)\Gamma\left(\alpha+1\right)}\frac{N!}{N^{n+1}(N-n-1)!}.
\end{align*}
For each $f\in\mathcal{C}^{n+1}\left[-1,1\right]$ with $n+1\leq n(\alpha,N)$ holds
\begin{align}\label{theoremgleichmaessigekonvergenz1ae}
\sup\limits_{x\in[-1,1]}{\left\lvert f(x)-\sum\limits_{k=0}^n{\frac{\left<f,Q_k\right>_\omega}{\left<Q_k,Q_k\right>_\omega}Q_k\left(\frac{N}{2}(1+x)\right)}\right\rvert}\leq D_{n,N}\sup\limits_{x\in[-1,1]}{\left\lvert f^{(n+1)}(x)\right\rvert}.
\end{align}
This estimation is not improvable in this sense, that the constant $D_{n,N}$ in inequality (\ref{theoremgleichmaessigekonvergenz1ae}) can not be replaced by a lower value under the above assumptions.
\end{theorem}

That provides a possibility to compare the directly and the classical method by consideration the maximum error (\glqq worst case\grqq) in the function classes
\begin{align}\label{klassefunkkne}
\mathcal{K}_{n+1}:=\left\{f\in\mathcal{C}^{n+1}\left[-1,1\right]:\ \sup\limits_{x\in[-1,1]}{\left\lvert f^{(n+1)}(x)\right\rvert\leq1}\right\}.
\end{align}
This maximum error is according to (\ref{theoremgleichmaessigekonvergenz1ae}) the constant $D_{n,N}$ and this is lower than in the corresponding classical case for $n+1\leq\frac{1}{2}-\alpha+\frac{1}{2}\sqrt{(2\alpha+1)(2\alpha+2N+1)}$ for each ratio $N/n$. That is proved in section \ref{VergleichzumKontinuierlichenFall}.
\par\medskip
In section \ref{Conclusions} we present further possible applications. For example the following result:
\par\medskip
Let $\alpha\geq0$, let
\begin{align*}
f\in K:=\left\{g\in\mathcal{C}^\infty\left[-1,1\right]:\ \lim_{n\to\infty}{\sup\limits_{x\in[-1,1]}{\left\lvert g^{(n)}(x)\right\rvert}\frac{n^{\alpha+\frac{1}{2}}}{2^nn!}}=0\right\}
\end{align*}
and let $(N_n)_{n\in\mathbb{N}}$ be a sequence with $N_n\geq2n(n+1)$. Then the method of least squares ${LS}_n^{N_n}[f]$ converges uniform on $[-1,1]$.
\par\medskip
We compare our approximation results with corresponding results for the continuous case in section \ref{Conclusions}. In the next section we demonstrate preliminary Lemmata and prove our main Theorem \ref{theoremgleichmaessigekonvergenz1e}.

\end{section}

\begin{section}{Preliminaries}
\label{Preliminaries}

For our investigations is the following result from H. Brass fundamental:

\begin{lemma}[cf. \cite{brass1984error}]\label{satzbrass1}
Let $\mathrm{d}\sigma$ a distribution on $\left[-1,1\right]$ and let
\begin{align*} 
\left\{q_k:\ k=0,\dots,n+1\right\}
\end{align*}
be a family of orthogonal polynomials, which are orthogonal with respect to the inner product
\begin{align*}
(f,g)_\sigma:=\int\limits_{-1}^1{f(x)g(x)\mathrm{d}\sigma(x)}.
\end{align*}
The polynomials are normalized by $\left(q_k,q_k\right)_\sigma=1$. Furthermore the distribution $\mathrm{d}\sigma$ satisfy the properties:
\begin{itemize}
  \item $\int_{-1}^1{f(x)\mathrm{d}\sigma(x)}=\int_{-1}^1{f(-x)\mathrm{d}\sigma(x)}$ f\"ur jedes $f\in\mathcal{C}\left[-1,1\right]$,
	\item $\sup\limits_{x\in[-1,1]}{\left\lvert q_k(x)\right\rvert}=q_k(1)$ f\"ur jedes $k=0,\dots,n+1$.
\end{itemize}
Let
\begin{align*}
C_n:=\frac{\sup\limits_{x\in[-1,1]}{\left\lvert q_{n+1}(x)\right\rvert}}{\sup\limits_{x\in[-1,1]}{\left\lvert q_{n+1}^{(n+1)}(x)\right\rvert}}.
\end{align*}
Then one has for each $f\in\mathcal{C}^{n+1}\left[-1,1\right]$
\begin{align}\label{satzbrass1a}
\sup\limits_{x\in[-1,1]}{\left\lvert f(x)-\sum\limits_{k=0}^n{\left(f,q_k\right)_\sigma q_k(x)}\right\rvert}\leq C_n\sup\limits_{x\in[-1,1]}{\left\lvert f^{(n+1)}(x)\right\rvert}.
\end{align}
This estimation is not improvable in this sense, that the constant $C_n$ in inequality (\ref{satzbrass1a}) can not be replaced by a lower value under the above assumptions.
\end{lemma}

To apply this result, we have to defined a corresponding distribution and the appropriate family of orthogonal polynomials in the following. First we use the representation of the Hahn polynomials by hypergeometric series: The Hahn polynomials $Q_n\equiv Q_n\left(\cdot;\alpha,\beta,N\right)$ are classical discrete orthogonal polynomials on the interval $I=\left[0,N\right]$ of degree $n$. They can defined by the hypergeometric function as follows:

\begin{definition}[{cf., e.g., \cite[p. 204]{koekoek2010hypergeometric}}]\label{definitionhyphahn}
Let $\alpha,\beta>-1$ and let $N\in\mathbb{N}_0$. The polynomials $Q_n\equiv Q_n\left(\cdot;\alpha,\beta,N\right)$ which are defined by
\begin{align}
\begin{aligned}
Q_n\left(x;\alpha,\beta,N\right)&={}_3 F_2\left(\begin{matrix}-n,\ n+\alpha+\beta+1,\ -x\\ \alpha+1,\ -N\end{matrix};1\right)\\
&=\sum\limits_{k=0}^n{\frac{\left(-n\right)_k\left(n+\alpha+\beta+1\right)_k\left(-x\right)_k}{\left(\alpha+1\right)_k\left(-N\right)_k}\frac{1}{k!}},
\end{aligned}
\end{align}
for each $n=0,\dots,N$, are said to be \textbf{Hahn polynomials}.
\end{definition}

The first Lemma give us a ratio $N/n$ for the boundedness of the Hahn polynomials. Furthermore we can see, that the maximum is on the boundary.

\begin{lemma}\label{lemmahahns1}
Let $\alpha>-\frac{1}{2}$ and let for $N\in\N$
\begin{align*}
n(\alpha,N):=\frac{1}{2}-\alpha+\frac{1}{2}\sqrt{(2\alpha+1)(2\alpha+2N+1)}.
\end{align*}
Then, for any $n\leq n(\alpha,N)$ holds
\begin{align}\label{lemmahahns1a}
\max_{x\in[0,N]}{\left\lvert Q_n(x;\alpha,\alpha,N)\right\rvert}=Q_n(0;\alpha,\alpha,N)=(-1)^nQ_n(N;\alpha,\alpha,N)=1.
\end{align}
\end{lemma}

\begin{proof}
It follows directly from \cite{dette1995new}
\begin{align*}
\max_{x\in[0,N]}{\left\lvert Q_n(x;\alpha,\alpha,N)\right\rvert}=\max\left\{\left\lvert Q_n(0;\alpha,\alpha,N)\right\rvert,\left\lvert Q_n(N;\alpha,\alpha,N)\right\rvert\right\}=1.
\end{align*}
Furthermore one has the following symmetries (cf., e.g., \cite{dette1995new}):
\begin{align*}
Q_n(0;\alpha,\alpha,N)=(-1)^nQ_n(N;\alpha,\alpha,N).
\end{align*}
With the definition \ref{definitionhyphahn} of the Hahn polynomials we have the positivity of\linebreak $Q_n\left(0;\alpha,\alpha,N\right)$:
\begin{align*}
Q_n\left(0;\alpha,\alpha,N\right)=\sum\limits_{k=0}^n{\frac{\left(-n\right)_k\left(n+\alpha+\alpha+1\right)_k\left(0\right)_k}{\left(\alpha+1\right)_k\left(-N\right)_k}\frac{1}{k!}}=1.
\end{align*}
\end{proof}

\begin{anmerkung}\label{anmerkunghahn2}
In the following let
\begin{align*}
\hat{Q}_k(x):=\frac{(-1)^kQ_k\left(\frac{N}{2}(1+x);\alpha,\alpha,N\right)}{\sqrt{\left<Q_k(\cdot;\alpha,\alpha,N),Q_k(\cdot;\alpha,\alpha,N)\right>_\omega}}.
\end{align*}
Let $N\in\mathbb{N}$ and let $\alpha=\beta>-\frac{1}{2}$. Furthermore we consider in this section the distribution $\mathrm{d}\sigma$ on $\left[-1,1\right]$, defined by
\begin{align*}
\int\limits_{-1}^1{f(x)\mathrm{d}\sigma(x)}=\sum\limits_{i=0}^N{f\left(-1+\frac{2i}{N}\right)\omega(i)}.
\end{align*}
\end{anmerkung}

In the next Lemma we prove all properties of the polynomials $\hat{Q}_k$ in Lemma \ref{satzbrass1}.

\begin{lemma}\label{lemmadistributions2}
Let $\alpha>-\frac{1}{2}$ and let for $N\in\N$
\begin{align*}
n(\alpha,N):=\frac{1}{2}-\alpha+\frac{1}{2}\sqrt{(2\alpha+1)(2\alpha+2N+1)}.
\end{align*}
Let $N\in\mathbb{N}$ and let $\mathrm{d}\sigma$ be the distribution of Remark \ref{anmerkunghahn2}. Furthermore let\linebreak $\left\{\hat{Q}_k:\ k=0,\dots,N\right\}$ be the family of polynomials, which is defined in Remark \ref{anmerkunghahn2}. Then one has with $n+1\leq n(\alpha,N)$ the following properties:
\begin{enumerate}
  \item $\left\{\hat{Q}_k:\ k=0,\dots,n+1\right\}$ is a family of orthogonal polynomials, which is orthogonal with respect to the inner product $(f,g)_\sigma:=\int_{-1}^1{f(x)g(x)\mathrm{d}\sigma(x)}$.
  \item $\left(\hat{Q}_k,\hat{Q}_k\right)_\sigma=1$.
  \item $\int_{-1}^1{f(x)\mathrm{d}\sigma(x)}=\int_{-1}^1{f(-x)\mathrm{d}\sigma(x)}$ for each $f\in\mathcal{C}\left[-1,1\right]$.
	\item $\sup\limits_{x\in[-1,1]}{\left\lvert \hat{Q}_k(x)\right\rvert}=\hat{Q}_k(1)$ for each $k=0,\dots,n+1$.
\end{enumerate}
\end{lemma}

\begin{proof}
For each $k,l\in\left\{1,\dots,n+1\right\}$ holds
\begin{align*}
\left(\hat{Q}_k,\hat{Q}_l\right)_\sigma=\int\limits_{-1}^1{\hat{Q}_k(x)\hat{Q}_l(x)\mathrm{d}\sigma(x)}=\sum\limits_{i=0}^N{\hat{Q}_k\left(-1+\frac{2i}{N}\right)\hat{Q}_l\left(-1+\frac{2i}{N}\right)\omega(i)}.
\end{align*}
With the definition of $\left\{\hat{Q}_k:\ k=0,\dots,N\right\}$ one has
\begin{align*}
\left(\hat{Q}_k,\hat{Q}_l\right)_\sigma=\frac{(-1)^{k+l}\left<Q_k,Q_l\right>_\omega}{\sqrt{\left<Q_k,Q_k\right>_\omega\left<Q_l,Q_l\right>_\omega}}.
\end{align*}
We obtain property $1$.\\
With $k=l$ holds $\left(\hat{Q}_k,\hat{Q}_k\right)_\sigma=1$ and we obtain property $2$.\\
Furthermore one has for any $f\in\mathcal{C}\left[-1,1\right]$
\begin{align*}
\int\limits_{-1}^1{f(x)\mathrm{d}\sigma(x)}=\sum\limits_{i=0}^N{f\left(-1+\frac{2i}{N}\right)\binom{\alpha+i}{i}\binom{\alpha+N-i}{N-i}}.
\end{align*}
With the index transformation $i\mapsto N-i$ and the equation
\begin{align*}
-1+\frac{2(N-i)}{N}=1-\frac{2i}{N},
\end{align*}
follows
\begin{align*}
\int\limits_{-1}^1{f(x)\mathrm{d}\sigma(x)}=\sum\limits_{i=0}^N{f\left(1-\frac{2i}{N}\right)\binom{\alpha+N-i}{N-i}\binom{\alpha+i}{i}}=\int\limits_{-1}^1{f(-x)\mathrm{d}\sigma(x)}.
\end{align*}
We obtain property $3$.\\
With Lemma \ref{lemmahahns1} one has for each $k=0,\dots,n+1$
\begin{align*}
\sup\limits_{x\in[-1,1]}{\left\lvert \hat{Q}_k(x)\right\rvert}=\sup\limits_{x\in[0,N]}{\left\lvert Q_k\left(x\right)\right\rvert}\frac{1}{\sqrt{\left<Q_k,Q_k\right>_\omega}}=\frac{(-1)^kQ_k\left(N\right)}{\sqrt{\left<Q_k,Q_k\right>_\omega}}=\hat{Q}_k(1),
\end{align*}
whereby we obtain property $4$.
\end{proof}

Now we can apply Lemma \ref{satzbrass1}. 

\begin{lemma}\label{lemmadistributions3}
Let $\alpha>-\frac{1}{2}$ and let for $N\in\N$
\begin{align*}
n(\alpha,N):=\frac{1}{2}-\alpha+\frac{1}{2}\sqrt{(2\alpha+1)(2\alpha+2N+1)}.
\end{align*}
Let $N\in\mathbb{N}$ and let $\mathrm{d}\sigma$ be the distribution of Remark \ref{anmerkunghahn2}. Furthermore let\linebreak $\left\{\hat{Q}_k:\ k=0,\dots,N\right\}$ be the family of polynomials, which is defined in Remark \ref{anmerkunghahn2}. Let
\begin{align}\label{lemmadistributions3a}
D_{n,N}:=\frac{\sup\limits_{x\in[-1,1]}{\left\lvert \hat{Q}_{n+1}(x)\right\rvert}}{\sup\limits_{x\in[-1,1]}{\left\lvert \hat{Q}_{n+1}^{(n+1)}(x)\right\rvert}}.
\end{align}
Then one has for each $f\in\mathcal{C}^{n+1}\left[-1,1\right]$ with $n+1\leq n(\alpha,N)$
\begin{align}\label{lemmadistributions3b}
\sup\limits_{x\in[-1,1]}{\left\lvert f(x)-\sum\limits_{k=0}^n{\left(f,\hat{Q}_k\right)_\sigma \hat{Q}_k(x)}\right\rvert}\leq D_{n,N}\sup\limits_{x\in[-1,1]}{\left\lvert f^{(n+1)}(x)\right\rvert}.
\end{align}
This estimation is not improvable in this sense, that the constant $D_{n,N}$ in inequality (\ref{lemmadistributions3b}) can not be replaced by a lower value under the above assumptions.
\end{lemma}

\begin{proof}
We apply Lemma \ref{satzbrass1} and Lemma \ref{lemmadistributions2}. With Lemma \ref{lemmadistributions2} the family of ortho-\linebreak gonal polynomials $\left\{\hat{Q}_k:\ k=0,\dots,N\right\}$ satisfy all the assumptions of Lemma \ref{satzbrass1}. So we can apply Lemma \ref{satzbrass1} and the claim follows.
\end{proof}

In the following Lemma we determine the factor $D_{n,N}$ in equation (\ref{lemmadistributions3a}) of the previous Lemma \ref{lemmadistributions3}.

\begin{lemma}\label{lemmadistributions4}
With the assumptions of Lemma \ref{lemmadistributions3} we have for any $n+1\leq n(\alpha,N)$
\begin{align}\label{lemmadistributions4a}
D_{n,N}=\frac{\sup\limits_{x\in[-1,1]}{\left\lvert \hat{Q}_{n+1}(x)\right\rvert}}{\sup\limits_{x\in[-1,1]}{\left\lvert \hat{Q}_{n+1}^{(n+1)}(x)\right\rvert}}=\frac{2^{n+1}\Gamma\left(n+2\alpha+2\right)\Gamma\left(n+\alpha+2\right)}{(n+1)!\Gamma\left(2n+2\alpha+3\right)\Gamma\left(\alpha+1\right)}\frac{N!}{N^{n+1}(N-n-1)!}.
\end{align}
\end{lemma}

\begin{proof}
First one has for any $x\in[-1,1]$ and any $n+1\leq n(\alpha,N)$
\begin{align*}
\begin{aligned}
\hat{Q}_{n+1}^{(n+1)}(x)&=\frac{(-1)^{n+1}}{\sqrt{\left<Q_{n+1},Q_{n+1}\right>_\omega}}\frac{\mathrm{d}^{n+1}}{\mathrm{d}x^{n+1}}Q_{n+1}\left(\frac{N}{2}(1+x)\right)\\
&=\frac{(-1)^{n+1}}{\sqrt{\left<Q_{n+1},Q_{n+1}\right>_\omega}}\left(\frac{N}{2}\right)^{n+1}Q_{n+1}^{(n+1)}\left(\frac{N}{2}(1+x)\right).
\end{aligned}
\end{align*}
With Lemma \ref{lemmahahns1} holds
\begin{align*}
\sup\limits_{x\in[-1,1]}{\left\lvert \hat{Q}_{n+1}(x)\right\rvert}=\frac{\sup\limits_{x\in[0,N]}{\left\lvert Q_{n+1}(x)\right\rvert}}{\sqrt{\left<Q_{n+1},Q_{n+1}\right>_\omega}}=\frac{1}{\sqrt{\left<Q_{n+1},Q_{n+1}\right>_\omega}}.
\end{align*}
Then one has
\begin{align}\label{lemmadistributions4aa}
\frac{\sup\limits_{x\in[-1,1]}{\left\lvert \hat{Q}_{n+1}(x)\right\rvert}}{\sup\limits_{x\in[-1,1]}{\left\lvert \hat{Q}_{n+1}^{(n+1)}(x)\right\rvert}}=\left(\frac{2}{N}\right)^{n+1}\frac{1}{\sup\limits_{x\in[0,N]}{\left\lvert Q_{n+1}^{(n+1)}(x)\right\rvert}}.
\end{align}
With the representation of the Hahn polynomials in definition \ref{definitionhyphahn} follows
\begin{align*}
Q_{n+1}^{(n+1)}(x)=\frac{\left(-n-1\right)_{n+1}\left(n+2\alpha+2\right)_{n+1}}{\left(\alpha+1\right)_{n+1}\left(-N\right)_{n+1}}\frac{1}{(n+1)!}\frac{\mathrm{d}^{n+1}}{\mathrm{d}x^{n+1}}\left(-x\right)_{n+1}.
\end{align*}
The term $\left(-x\right)_{n+1}$ is a polynomial of degree $n+1$, which we can write in the form
\begin{align*}
\left(-x\right)_{n+1}=(-x)(-x+1)\cdot\ldots\cdot(-x+n)=(-1)^{n+1}x^{n+1}+\tilde{p}(x)
\end{align*}
Hereby is $\tilde{p}\in\mathcal{P}_n$ a polynomial of degree $n$. Hence we have
\begin{align*}
Q_{n+1}^{(n+1)}(x)=\frac{\left(-n-1\right)_{n+1}\left(n+2\alpha+2\right)_{n+1}}{\left(\alpha+1\right)_{n+1}\left(-N\right)_{n+1}}\frac{1}{(n+1)!}(-1)^{n+1}(n+1)!.
\end{align*}
With the transformations
\begin{align*}
\left(-n-1\right)_{n+1}=(-n-1)(-n)\cdot\ldots\cdot(-1)=(-1)^{n+1}(n+1)!
\end{align*}
and
\begin{align*}
\left(-N\right)_{n+1}=(-N)(-N+1)\cdot\ldots\cdot(-N+n)=(-1)^{n+1}\frac{N!}{(N-n-1)!}
\end{align*}
we obtain
\begin{align*}
\begin{aligned}
Q_{n+1}^{(n+1)}(x)&=\frac{(-1)^{n+1}(n+1)!\left(n+2\alpha+2\right)_{n+1}(N-n-1)!}{\left(\alpha+1\right)_{n+1}(-1)^{n+1}N!}(-1)^{n+1}\\
&=\frac{(-1)^{n+1}(n+1)!\Gamma\left(2n+2\alpha+3\right)\Gamma\left(\alpha+1\right)(N-n-1)!}{\Gamma\left(n+2\alpha+2\right)\Gamma\left(n+\alpha+2\right)N!}.
\end{aligned}
\end{align*}
Enter into the equation (\ref{lemmadistributions4aa}),
\begin{align*}
\frac{\sup\limits_{x\in[-1,1]}{\left\lvert \hat{Q}_{n+1}(x)\right\rvert}}{\sup\limits_{x\in[-1,1]}{\left\lvert \hat{Q}_{n+1}^{(n+1)}(x)\right\rvert}}=\left(\frac{2}{N}\right)^{n+1}\frac{\Gamma\left(n+2\alpha+2\right)\Gamma\left(n+\alpha+2\right)N!}{(n+1)!\Gamma\left(2n+2\alpha+3\right)\Gamma\left(\alpha+1\right)(N-n-1)!},
\end{align*}
and it follows equation (\ref{lemmadistributions4a}).
\end{proof}

Now we can prove the main Theorem \ref{theoremgleichmaessigekonvergenz1e} with the aid of the previous lemmata:

\begin{proof}
Let $N\in\mathbb{N}$ and let $f\in\mathcal{C}^{n+1}\left[-1,1\right]$ with $n+1\leq n(\alpha,N)$. Then one has for each $x\in[-1,1]$
\begin{align*}
\sum\limits_{k=0}^n{\left(f,\hat{Q}_k\right)_\sigma \hat{Q}_k(x)}=\sum\limits_{k=0}^n{\frac{\left<f,Q_k\right>_\omega}{\left<Q_k,Q_k\right>_\omega}Q_k\left(\frac{N}{2}(1+x)\right)}.
\end{align*}
Now we apply Lemma \ref{lemmadistributions3} and Lemma \ref{lemmadistributions4} and we obtain
\begin{align*}
\begin{aligned}
&\sup\limits_{x\in[-1,1]}{\left\lvert f(x)-\sum\limits_{k=0}^n{\frac{\left<f,Q_k\right>_\omega}{\left<Q_k,Q_k\right>_\omega}Q_k\left(\frac{N}{2}(1+x)\right)}\right\rvert}\\
\leq&\sup\limits_{x\in[-1,1]}{\left\lvert f^{(n+1)}(x)\right\rvert}\frac{2^{n+1}\Gamma\left(n+2\alpha+2\right)\Gamma\left(n+\alpha+2\right)}{(n+1)!\Gamma\left(2n+2\alpha+3\right)\Gamma\left(\alpha+1\right)}\frac{N!}{N^{n+1}(N-n-1)!}.
\end{aligned}
\end{align*}
We obtain with Lemma \ref{lemmadistributions3}, that the estimation is not improvable.
\end{proof}

\end{section}

\begin{section}{Conclusions}
\label{Conclusions}

In this section we present some results, which we obtain by use of Theorem \ref{theoremgleichmaessigekonvergenz1e}. Especially we discuss some cases, in which we obtain the uniform convergence of the method of least squares. First we investigate the factor $D_{n,N}$ of Theorem \ref{theoremgleichmaessigekonvergenz1e}.

\begin{subsection}{Uniform convergence of the discrete method of least squares}

At the beginning we give the following Lemmata. The $\Gamma$-function satisfy the asymptotic property:

\begin{lemma}[{cf., e.g., \cite[p. 257]{abramowitz1964handbook}}]\label{lemmagamma0}
For $a, b>0$ holds
\begin{align}\label{lemmagamma0a}
N^{b-a}\frac{\Gamma(N+a)}{\Gamma(N+b)}=1+\frac{(a-b)(a+b-1)}{2N}+\mathcal{O}\left(\frac{1}{N^2}\right).
\end{align}
\end{lemma}

\begin{lemma}\label{lemmagleichkors5}
Let $\alpha>-\frac{1}{2}$. Then one has
\begin{align}
\frac{\Gamma\left(n+2\alpha+2\right)\Gamma\left(n+\alpha+2\right)}{\Gamma\left(2n+2\alpha+3\right)}=\frac{(n+1)!(n+1)!}{(2n+2)!}\frac{n^\alpha}{2^{2\alpha}}\left(1+\mathcal{O}\left(n^{-1}\right)\right).
\end{align}
\end{lemma}

\begin{proof}
We use Lemma \ref{lemmagamma0} and obtain
\begin{align*}
\begin{aligned}
\frac{\Gamma\left(n+2\alpha+2\right)\Gamma\left(n+\alpha+2\right)}{\Gamma\left(2n+2\alpha+3\right)}&=\frac{\Gamma(n+1)n^{(2\alpha+2)-1}\Gamma(n+1)n^{(\alpha+2)-1}}{\Gamma(2n+1)(2n)^{(2\alpha+3)-1}}\left(1+\mathcal{O}\left(n^{-1}\right)\right)\\
&=\frac{n!n!}{(2n)!}\frac{n^\alpha}{2^{2\alpha+2}}\left(1+\mathcal{O}\left(n^{-1}\right)\right)\\
&=\frac{(n+1)!(n+1)!}{(2n+2)!}\frac{n^\alpha}{2^{2\alpha}}\left(1+\mathcal{O}\left(n^{-1}\right)\right).
\end{aligned}
\end{align*}
\end{proof}

\begin{lemma}\label{lemmagleichkors6}
One has
\begin{align}
\frac{\sqrt{\pi n}}{2^nn!}\e^{\frac{2}{12n+1}-\frac{1}{24n}}\leq\frac{2^nn!}{(2n)!}\leq\frac{\sqrt{\pi n}}{2^nn!}\e^{\frac{1}{6n}-\frac{1}{24n+1}}.
\end{align}
\end{lemma}

\begin{proof}
We prove both inequalities successively. For that we use the Stirling's formula (cf., e.g., \cite[p. 50-53]{feller1968introduction}, \cite{robbins1955remark})
\begin{align*}
\e^{\frac{1}{12n+1}}\leq\frac{n!}{\sqrt{2\pi n}\left(\frac{n}{\e}\right)^n}\leq \e^{\frac{1}{12n}}.
\end{align*}
First we show the left inequality. With the aid of the Stirling's formula we have
\begin{align*}
\begin{aligned}
\frac{2^nn!}{(2n)!}&\geq 2^n\frac{\sqrt{2\pi n}\left(\frac{n}{\e}\right)^n}{\sqrt{2\pi (2n)}\left(\frac{2n}{\e}\right)^{2n}}\frac{\e^{\frac{1}{12n+1}}}{\e^{\frac{1}{24n}}}\\
&=\frac{\sqrt{\pi n}}{2^n}\frac{1}{\sqrt{2\pi n}\left(\frac{n}{\e}\right)^n}\frac{\e^{\frac{1}{12n+1}}}{\e^{\frac{1}{24n}}}.
\end{aligned}
\end{align*}
We use one more time the Stirling's formula
\begin{align*}
\frac{2^nn!}{(2n)!}\geq\frac{\sqrt{\pi n}}{2^n}\frac{1}{n!}\frac{\e^{\frac{1}{12n+1}}\e^{\frac{1}{12n+1}}}{\e^{\frac{1}{24n}}}=\frac{\sqrt{\pi n}}{2^nn!}\e^{\frac{2}{12n+1}-\frac{1}{24n}}.
\end{align*}
We obtain the left inequality. For the right inequality we have by use of the Stirling's formula again
\begin{align*}
\begin{aligned}
\frac{2^nn!}{(2n)!}&\leq 2^n\frac{\sqrt{2\pi n}\left(\frac{n}{\e}\right)^n}{\sqrt{2\pi (2n)}\left(\frac{2n}{\e}\right)^{2n}}\frac{\e^{\frac{1}{12n}}}{\e^{\frac{1}{24n+1}}}\\
&=\frac{\sqrt{\pi n}}{2^n}\frac{1}{\sqrt{2\pi n}\left(\frac{n}{\e}\right)^n}\frac{\e^{\frac{1}{12n}}}{\e^{\frac{1}{24n+1}}}.
\end{aligned}
\end{align*}
We use one more time the Stirling's formula
\begin{align*}
\frac{2^nn!}{(2n)!}\leq\frac{\sqrt{\pi n}}{2^n}\frac{1}{n!}\frac{\e^{\frac{1}{12n}}\e^{\frac{1}{12n}}}{\e^{\frac{1}{24n+1}}}=\frac{\sqrt{\pi n}}{2^nn!}\e^{\frac{1}{6n}-\frac{1}{24n+1}}.
\end{align*}
We obtain the right inequality.
\end{proof}

\begin{anmerkung}\label{anmerkunggleichkors6a}
With Lemma \ref{lemmagleichkors6} we have
\begin{align}
\frac{2^nn!}{(2n)!}=\frac{\sqrt{\pi n}}{2^nn!}\left(1+\mathcal{O}\left(n^{-1}\right)\right).
\end{align}
\end{anmerkung}

Now we can simplify the estimation in Theorem \ref{theoremgleichmaessigekonvergenz1e}.

\begin{korollar}\label{korollargleichmaessig1}
Let $\alpha>-\frac{1}{2}$ and let for $N\in\N$
\begin{align*}
n(\alpha,N):=\frac{1}{2}-\alpha+\frac{1}{2}\sqrt{(2\alpha+1)(2\alpha+2N+1)}.
\end{align*}
Then one has for each $f\in\mathcal{C}^{n+1}\left[-1,1\right]$
\begin{align}\label{korollargleichmaessig1a}
\begin{aligned}
&\sup\limits_{x\in[-1,1]}{\left\lvert f(x)-\sum\limits_{k=0}^n{\frac{\left<f,Q_k\right>_\omega}{\left<Q_k,Q_k\right>_\omega}Q_k\left(\frac{N}{2}(1+x)\right)}\right\rvert}\\
\leq&\sup\limits_{x\in[-1,1]}{\left\lvert f^{(n+1)}(x)\right\rvert}\frac{\sqrt{\pi n}}{2^{n+1}(n+1)!}\cdot\frac{n^\alpha}{\Gamma\left(\alpha+1\right)2^{2\alpha}}\left(1+\mathcal{O}\left(n^{-1}\right)\right),
\end{aligned}
\end{align}
with $n+1\leq n(\alpha,N)$.
\end{korollar}

\begin{proof}
Let $N\in\mathbb{N}$ and let $f\in\mathcal{C}^{n+1}\left[-1,1\right]$ with $n+1\leq n(\alpha,N)$. First we have
\begin{align*}
\frac{N!}{N^{n+1}(N-n-1)!}=\prod_{i=0}^n{\left(1-\frac{i}{N}\right)}\leq1.
\end{align*}
With Lemma \ref{lemmagleichkors5} we obtain
\begin{align*}
\begin{aligned}
&\frac{2^{n+1}\Gamma\left(n+2\alpha+2\right)\Gamma\left(n+\alpha+2\right)}{(n+1)!\Gamma\left(2n+2\alpha+3\right)\Gamma\left(\alpha+1\right)}\frac{N!}{N^{n+1}(N-n-1)!}\\
\leq&\frac{2^{n+1}\Gamma\left(n+2\alpha+2\right)\Gamma\left(n+\alpha+2\right)}{(n+1)!\Gamma\left(2n+2\alpha+3\right)\Gamma\left(\alpha+1\right)}\\
=&\frac{2^{n+1}}{(n+1)!\Gamma\left(\alpha+1\right)}\frac{(n+1)!(n+1)!}{(2n+2)!}\frac{n^\alpha}{2^{2\alpha}}\left(1+\mathcal{O}\left(n^{-1}\right)\right)\\
=&\frac{2^{n+1}(n+1)!}{(2n+2)!}\frac{n^\alpha}{\Gamma\left(\alpha+1\right)2^{2\alpha}}\left(1+\mathcal{O}\left(n^{-1}\right)\right).
\end{aligned}
\end{align*}
We apply Remark \ref{anmerkunggleichkors6a} and obtain
\begin{align*}
\begin{aligned}
&\frac{2^{n+1}\Gamma\left(n+2\alpha+2\right)\Gamma\left(n+\alpha+2\right)}{(n+1)!\Gamma\left(2n+2\alpha+3\right)\Gamma\left(\alpha+1\right)}\frac{N!}{N^{n+1}(N-n-1)!}\\
\leq&\frac{\sqrt{\pi (n+1)}}{2^{n+1}(n+1)!}\frac{n^\alpha}{\Gamma\left(\alpha+1\right)2^{2\alpha}}\left(1+\mathcal{O}\left(n^{-1}\right)\right).
\end{aligned}
\end{align*}
We use Theorem \ref{theoremgleichmaessigekonvergenz1e}, then we obtain the inequality (\ref{korollargleichmaessig1a}).
\end{proof}

For the important case $\alpha=0$ we complement the following estimation.

\begin{korollar}\label{korollargleichmaessig2}
Let for $N\in\N$
\begin{align*}
n(N):=\frac{1}{2}+\frac{1}{2}\sqrt{2N+1}.
\end{align*}
Furthermore let
\begin{align*}
D_n:=\frac{\sqrt{\pi (n+1)}}{2^{n+1}(n+1)!}\e^{\frac{1}{6(n+1)}-\frac{1}{24(n+1)+1}}.
\end{align*}
Then one has for each $f\in\mathcal{C}^{n+1}\left[-1,1\right]$ and for any $N\in\mathbb{N}$ with $n+1\leq n(N)$
\begin{align}\label{korollargleichmaessig2a}
\sup\limits_{x\in[-1,1]}{\left\lvert f(x)-\sum\limits_{k=0}^n{\frac{\left<f,Q_k\right>_\omega}{\left<Q_k,Q_k\right>_\omega}Q_k\left(\frac{N}{2}(1+x)\right)}\right\rvert}\leq D_n\sup\limits_{x\in[-1,1]}{\left\lvert f^{(n+1)}(x)\right\rvert}.
\end{align}
Under the above assumptions is the constant $D_n$ in inequality (\ref{korollargleichmaessig2a}) improvable at most by the factor
\begin{align*}
d_n:=\e^{\frac{2}{12(n+1)+1}+\frac{1}{24(n+1)+1}-\frac{1}{6(n+1)}-\frac{1}{24(n+1)}}\approx 1.
\end{align*}
\end{korollar}

\begin{proof}
Let $N\in\mathbb{N}$ and let $f\in\mathcal{C}^{n+1}\left[-1,1\right]$ with $n+1\leq n(N)$. For $\alpha=0$ reduce the Theorem \ref{theoremgleichmaessigekonvergenz1e} to
\begin{align*}
\begin{aligned}
&\sup\limits_{x\in[-1,1]}{\left\lvert f(x)-\sum\limits_{k=0}^n{\frac{\left<f,Q_k\right>_\omega}{\left<Q_k,Q_k\right>_\omega}Q_k\left(\frac{N}{2}(1+x)\right)}\right\rvert}\\
\leq&\sup\limits_{x\in[-1,1]}{\left\lvert f^{(n+1)}(x)\right\rvert}\frac{2^{n+1}\Gamma\left(n+2\right)\Gamma\left(n+2\right)}{(n+1)!\Gamma\left(2n+3\right)\Gamma\left(1\right)}\frac{N!}{N^{n+1}(N-n-1)!}\\
=&\sup\limits_{x\in[-1,1]}{\left\lvert f^{(n+1)}(x)\right\rvert}\frac{2^{n+1}(n+1)!}{(2n+2)!}\frac{N!}{N^{n+1}(N-n-1)!}.
\end{aligned}
\end{align*}
With the estimation
\begin{align*}
\frac{N!}{N^{n+1}(N-n-1)!}=\prod_{i=0}^n{\left(1-\frac{i}{N}\right)}\leq1,
\end{align*}
we obtain
\begin{align*}
\begin{aligned}
&\sup\limits_{x\in[-1,1]}{\left\lvert f(x)-\sum\limits_{k=0}^n{\frac{\left<f,Q_k\right>_\omega}{\left<Q_k,Q_k\right>_\omega}Q_k\left(\frac{N}{2}(1+x)\right)}\right\rvert}\\
\leq&\sup\limits_{x\in[-1,1]}{\left\lvert f^{(n+1)}(x)\right\rvert}\frac{2^{n+1}(n+1)!}{(2n+2)!}.
\end{aligned}
\end{align*}
This estimation is for any $N$ not improvable because of
\begin{align*}
\frac{N!}{N^{n+1}(N-n-1)!}=\prod_{i=0}^n{\left(1-\frac{i}{N}\right)}\rightarrow1,\ \mbox{where}\ N\to\infty.
\end{align*}
With Lemma \ref{lemmagleichkors6} we have
\begin{align*}
\begin{aligned}
&\sup\limits_{x\in[-1,1]}{\left\lvert f(x)-\sum\limits_{k=0}^n{\frac{\left<f,Q_k\right>_\omega}{\left<Q_k,Q_k\right>_\omega}Q_k\left(\frac{N}{2}(1+x)\right)}\right\rvert}\\
\leq&\sup\limits_{x\in[-1,1]}{\left\lvert f^{(n+1)}(x)\right\rvert}\frac{\sqrt{\pi (n+1)}}{2^{n+1}(n+1)!}\e^{\frac{1}{6(n+1)}-\frac{1}{24(n+1)+1}}
\end{aligned}
\end{align*}
and we have that the inequality is not improvable except for the factor
\begin{align*}
d_n=\e^{\frac{2}{12(n+1)+1}+\frac{1}{24(n+1)+1}-\frac{1}{6(n+1)}-\frac{1}{24(n+1)}}.
\end{align*}
\end{proof}

With this Corollary \ref{korollargleichmaessig2} we can give a special and interest answer to our initially question:\\
For which classes of functions $K\subset\mathcal{C}\left[-1,1\right]$ and which ratio $N/n$ converges the method of least squares $\left({LS}_n^N\right)$ uniformly?

\begin{korollar}\label{korollargleichmaessig3}
Let $\alpha>-\frac{1}{2}$, let
\begin{align*}
f\in K:=\left\{g\in\mathcal{C}^\infty\left[-1,1\right]:\ \lim_{n\to\infty}{\sup\limits_{x\in[-1,1]}{\left\lvert g^{(n)}(x)\right\rvert}\frac{n^{\alpha+\frac{1}{2}}}{2^nn!}}=0\right\}
\end{align*}
and let $(N_n)_{n\in\mathbb{N}}$ be a sequence with 
\begin{align*}
N_n\geq\frac{2n^2+\left(4\alpha+2\right)n}{2\alpha+1}.
\end{align*}
Then the method of least squares ${LS}_n^{N_n}[f]$ converges uniform on the interval $[-1,1]$.
\end{korollar}

\begin{proof}
First one has
\begin{align*}
N_n\geq\frac{2n^2+\left(4\alpha+2\right)n}{2\alpha+1}=\frac{2\left(n+\frac{1}{2}+\alpha\right)^2}{2\alpha+1}-\frac{2\alpha+1}{2}.
\end{align*}
With simple transformations holds
\begin{align*}
\left(2\alpha+1\right)\left(2\alpha+2N_n+1\right)\geq4\left(n+\frac{1}{2}+\alpha\right)^2.
\end{align*}
We transform again and obtain
\begin{align*}
n\left(\alpha,N_n\right):=\frac{1}{2}-\alpha+\frac{1}{2}\sqrt{\left(2\alpha+1\right)\left(2\alpha+2N_n+1\right)}\geq n+1.
\end{align*}
Now we can use Corollary \ref{korollargleichmaessig1}:
\begin{align*}
\begin{aligned}
&\lim\limits_{n\to\infty}{\sup\limits_{x\in[-1,1]}{\left\lvert f(x)-\sum\limits_{k=0}^n{\frac{\left<f,Q_k\right>_\omega}{\left<Q_k,Q_k\right>_\omega}Q_k\left(\frac{N_n}{2}(1+x)\right)}\right\rvert}}\\
\leq&\lim\limits_{n\to\infty}{\sup\limits_{x\in[-1,1]}{\left\lvert f^{(n+1)}(x)\right\rvert}\frac{\sqrt{\pi n}}{2^{n+1}(n+1)!}\cdot\frac{n^\alpha}{\Gamma\left(\alpha+1\right)2^{2\alpha}}\left(1+\mathcal{O}\left(n^{-1}\right)\right)}\\
\leq&\frac{\sqrt{\pi}}{\Gamma\left(\alpha+1\right)2^{2\alpha}}\lim\limits_{n\to\infty}{\sup\limits_{x\in[-1,1]}{\left\lvert f^{(n)}(x)\right\rvert}\frac{n^{\alpha+\frac{1}{2}}}{2^nn!}}.
\end{aligned}
\end{align*}
Because of $f\in K$ one has
\begin{align*}
\lim\limits_{n\to\infty}{\sup\limits_{x\in[-1,1]}{\left\lvert f(x)-\sum\limits_{k=0}^n{\frac{\left<f,Q_k\right>_\omega}{\left<Q_k,Q_k\right>_\omega}Q_k\left(\frac{N_n}{2}(1+x)\right)}\right\rvert}}=0.
\end{align*}
\end{proof}

Concerning the above question we can easily give a sequence $(N_n)_{n\in\mathbb{N}}$ independent of $\alpha$:

\begin{korollar}\label{korollargleichmaessig4}
Let $\alpha\geq0$, let
\begin{align*}
f\in K:=\left\{g\in\mathcal{C}^\infty\left[-1,1\right]:\ \lim_{n\to\infty}{\sup\limits_{x\in[-1,1]}{\left\lvert g^{(n)}(x)\right\rvert}\frac{n^{\alpha+\frac{1}{2}}}{2^nn!}}=0\right\}
\end{align*}
and let $(N_n)_{n\in\mathbb{N}}$ be a sequence with $N_n\geq2n(n+1)$. Then the method of least squares ${LS}_n^{N_n}[f]$ converges uniform on the interval $[-1,1]$.
\end{korollar}

\begin{proof}
The sequence $(N_n)_{n\in\mathbb{N}}$ fulfils the assumption of Corollary \ref{korollargleichmaessig3} independent of $\alpha$. Because one has
\begin{align*}
N_n\geq2n(n+1)\geq2n^2+2n\geq\frac{2n^2}{2\alpha+1}+\frac{\left(2\alpha+1\right)2n}{2\alpha+1}\geq\frac{2n^2+\left(4\alpha+2\right)n}{2\alpha+1}.
\end{align*}
\end{proof}

\end{subsection}

\begin{subsection}{Comparison to the continuous case}
\label{VergleichzumKontinuierlichenFall}

In this subsection we compare our approximation results of the discrete method of least squares with the results of the continuous method. The continuous method is the series expansion of a function by Jacobi polynomials $P_n\equiv P_n^{\alpha,\alpha}$, then the least square operator ${LS}_n$ can be represented by
\begin{align}\label{darstellungmethodekont}
{LS}_n[f]=\sum\limits_{k=0}^n{\frac{\left(P_k,f\right)_\varrho}{\left(P_k,P_k\right)_\varrho}P_k}.
\end{align}
This case was investigated by H. Brass in \cite{brass1980approximation}. First we provide in the following some important properties of the Jacobi polynomials:
\par\medskip
The Jacobi polynomials $P_n\equiv P_n^{\alpha,\beta}$ are classical orthogonal polynomials on the interval $I=\left[-1,1\right]$ of degree $n$. They can defined by the hypergeometric function as follows:

\begin{definition}[{cf., e.g., \cite[p. 216]{koekoek2010hypergeometric}}]\label{definitionhypjacobi}
Let $\alpha,\beta>-1$. The polynomials $P_n\equiv P_n^{\alpha,\beta}$ which are defined by
\begin{align}
\begin{aligned}
P_n^{\alpha,\beta}(x)&=\frac{(\alpha+1)_n}{n!}{}_2 F_1\left(\begin{matrix}-n,\ n+\alpha+\beta+1\\ \alpha+1\end{matrix};\frac{1-x}{2}\right)\\
&=\frac{(\alpha+1)_n}{n!}\sum\limits_{k=0}^n{\frac{\left(-n\right)_k\left(n+\alpha+\beta+1\right)_k}{\left(\alpha+1\right)_k}\frac{(1-x)^k}{2^kk!}},
\end{aligned}
\end{align}
for each $n\in\mathbb{N}_0$, are said to be \textbf{Jacobi polynomials}.
\end{definition}

The Jacobi polynomials $P_n^{\alpha,\beta}$ are orthogonal on the interval $I=\left[-1,1\right]$ with respect to the inner product
\begin{align*}
(f,g)_\varrho:=\int\limits_{-1}^1{f(x)g(x)\varrho(x)\mathrm{d}x},
\end{align*}
where $\varrho$ is the weight-function given by
\begin{align*}
\varrho(x):=(1-x)^\alpha(1+x)^\beta.
\end{align*}
They are normalized by
\begin{align*}%\label{orthogjacobi1a}
\left(P_n^{\alpha,\beta},P_n^{\alpha,\beta}\right)_\varrho=\frac{2^{\alpha+\beta+1}\Gamma(n+\alpha+1)\Gamma(n+\beta+1)}{(2n+\alpha+\beta+1)n!\Gamma(n+\alpha+\beta+1)}
\end{align*}
(cf., e.g., \cite[p. 217]{koekoek2010hypergeometric}).
\par\medskip
For $\max\left\{\alpha,\beta\right\}\geq-\frac{1}{2}$ the Jacobi polynomials are bounded on the interval $[-1,1]$ as follows:
\begin{align}\label{satzmaxjacobi1a}
\max_{x\in[-1,1]}{\left\lvert P_n^{\alpha,\beta}(x)\right\rvert}=\binom{n+\max\left\{\alpha,\beta\right\}}{n}
\end{align}
(cf., e.g., \cite[p. 786]{abramowitz1964handbook}).
\par\medskip
H. Brass proved the following result:

\begin{lemma}[cf. \cite{brass1980approximation}]\label{satzbrass2}
Let $\alpha\geq-\frac{1}{2}$. Further let
\begin{align}\label{satzbrass2a}
C_n:=\frac{\sup\limits_{x\in[-1,1]}{\left\lvert P_{n+1}(x)\right\rvert}}{\sup\limits_{x\in[-1,1]}{\left\lvert P_{n+1}^{(n+1)}(x)\right\rvert}}.
\end{align}
Then one has for each $f\in\mathcal{C}^{n+1}\left[-1,1\right]$
\begin{align}\label{satzbrass2b}
\sup\limits_{x\in[-1,1]}{\left\lvert f(x)-\sum\limits_{k=0}^n{\frac{\left(f,P_k\right)_\varrho}{\left(P_k,P_k\right)_\varrho}P_k(x)}\right\rvert}\leq C_n\sup\limits_{x\in[-1,1]}{\left\lvert f^{(n+1)}(x)\right\rvert}.
\end{align}
This estimation is not improvable in this sense, that the constant $C_n$ in inequality (\ref{satzbrass2b}) can not be replaced by a lower value under the above assumptions.
\end{lemma}

This result follows also from Lemma \ref{satzbrass1} (cf. \cite{brass1984error}). In the following Lemma we determine the factor $C_n$ in equation (\ref{satzbrass2a}) of the previous Lemma \ref{satzbrass2}.

\begin{lemma}\label{lemmagleichkont1}
Let $\alpha\geq-\frac{1}{2}$. Then for the constant $C_n$ of Lemma \ref{satzbrass2} holds
\begin{align}\label{lemmagleichkont1a}
C_n=\frac{\sup\limits_{x\in[-1,1]}{\left\lvert P_{n+1}(x)\right\rvert}}{\sup\limits_{x\in[-1,1]}{\left\lvert P_{n+1}^{(n+1)}(x)\right\rvert}}=\frac{2^{n+1}\Gamma(n+\alpha+2)\Gamma(n+2\alpha+2)}{(n+1)!\Gamma\left(2n+2\alpha+3\right)\Gamma(\alpha+1)}.
\end{align}
\end{lemma}

\begin{proof}
First the Jacobi polynomials are given by Definition \ref{definitionhypjacobi}
\begin{align*}
P_{n+1}(x)=\frac{(\alpha+1)_{n+1}}{(n+1)!}\sum\limits_{k=0}^{n+1}{\frac{\left(-n-1\right)_k\left(n+2\alpha+2\right)_k}{\left(\alpha+1\right)_k}\frac{(1-x)^k}{2^kk!}}.
\end{align*}
We differentiate $(n+1)$ times and obtain
\begin{align*}
\begin{aligned}
P_{n+1}^{(n+1)}(x)&=\frac{(\alpha+1)_{n+1}}{(n+1)!}\frac{\left(-n-1\right)_{n+1}\left(n+2\alpha+2\right)_{n+1}}{\left(\alpha+1\right)_{n+1}2^{n+1}(n+1)!}\frac{\mathrm{d}^{n+1}}{\mathrm{d}x^{n+1}}(1-x)^{n+1}\\
&=\frac{\left(-n-1\right)_{n+1}\left(n+2\alpha+2\right)_{n+1}}{(n+1)!2^{n+1}(n+1)!}(-1)^{n+1}(n+1)!.
\end{aligned}
\end{align*}
With the transformations
\begin{align*}
\left(-n-1\right)_{n+1}=(-n-1)(-n)\cdot\ldots\cdot(-1)=(-1)^{n+1}(n+1)!
\end{align*}
we have
\begin{align*}
P_{n+1}^{(n+1)}(x)=\frac{(-1)^{n+1}(n+1)!\left(n+2\alpha+2\right)_{n+1}}{2^{n+1}(n+1)!}(-1)^{n+1}=\frac{\left(n+2\alpha+2\right)_{n+1}}{2^{n+1}}.
\end{align*}
Then one has
\begin{align*}
\sup\limits_{x\in[-1,1]}{\left\lvert P_{n+1}^{(n+1)}(x)\right\rvert}=\frac{\left(n+2\alpha+2\right)_{n+1}}{2^{n+1}}=\frac{\Gamma\left(2n+2\alpha+3\right)}{2^{n+1}\Gamma\left(n+2\alpha+2\right)}.
\end{align*}
With equation (\ref{satzmaxjacobi1a}) follows
\begin{align*}
\sup_{x\in[-1,1]}{\left\lvert P_{n+1}(x)\right\rvert}=\binom{n+1+\alpha}{n+1}=\frac{\Gamma(n+\alpha+2)}{\Gamma(n+2)\Gamma(\alpha+1)}.
\end{align*}
Then we have equation (\ref{lemmagleichkont1a}).
\end{proof}

Now we compare the constants $D_{n,N}$ of Theorem \ref{theoremgleichmaessigekonvergenz1e} (discrete case) and $C_n$ of Lemma \ref{satzbrass2} (continuous case), which are both not improvable. For the quotient $D_{n,N}/C_n$ one has with $\alpha>-\frac{1}{2}$
\begin{align*}
\frac{D_{n,N}}{C_n}=\frac{N!}{N^{n+1}(N-n-1)!}=\prod_{i=0}^n{\left(1-\frac{i}{N}\right)}\leq1,
\end{align*}
whereby in the discrete case we have the additional assumption $n+1\leq \frac{1}{2}-\alpha+\frac{1}{2}\sqrt{(2\alpha+1)(2\alpha+2N+1)}$. We define a function class $\mathcal{K}_n$ by
\begin{align}\label{klassefunkkn}
\mathcal{K}_n:=\left\{f\in\mathcal{C}^n\left[-1,1\right]:\ \sup\limits_{x\in[-1,1]}{\left\lvert f^{(n)}(x)\right\rvert\leq1}\right\},
\end{align}
then we obtain the following Corollary:

\begin{korollar}\label{korollargleichkont2}
Let $\alpha=\beta>-\frac{1}{2}$. Further let ${LS}_n$ be the continuous least square operator according to equation (\ref{darstellungmethodekont}) and let ${LS}_n^N$ be the discrete least square operator according to equation (\ref{darstellungmethodediskret}) with $n+1\leq \frac{1}{2}-\alpha+\frac{1}{2}\sqrt{(2\alpha+1)(2\alpha+2N+1)}$. Then one has
\begin{align}
\sup\limits_{f\in\mathcal{K}_{n+1}}{\sup\limits_{x\in[-1,1]}{\left\lvert f(x)-{LS}_n^N[f](x)\right\rvert}}=\prod_{i=0}^n{\left(1-\frac{i}{N}\right)}\sup\limits_{f\in\mathcal{K}_{n+1}}{\sup\limits_{x\in[-1,1]}{\left\lvert f(x)-{LS}_n[f](x)\right\rvert}}.
\end{align}
\end{korollar}

\begin{anmerkung}\label{anmerkunggleichkont3}
For the practical use we obtain for $n\in\mathbb{N}_0$ with Corollary \ref{korollargleichkont2} the following guarantee: The \glqq worst case\grqq\ respecting to the class $\mathcal{K}_{n+1}$ is in the continuous case worse than the corresponding discrete case, if the polynomial degree $n$ and the number of nodes $N+1$ fulfil the inequality $n+1\leq \frac{1}{2}-\alpha+\frac{1}{2}\sqrt{(2\alpha+1)(2\alpha+2N+1)}$.
\end{anmerkung}

\begin{anmerkung}\label{anmerkunggleichkont4}
If we consider the \glqq worst case\grqq again
\begin{align}\label{anmerkunggleichkont4a}
\sup\limits_{f\in\mathcal{K}_{n+1}}{\sup\limits_{x\in[-1,1]}{\left\lvert f(x)-{LS}_n^N[f](x)\right\rvert}},
\end{align}
we obtain:\\
A ratio ${n^k}/N\rightarrow 0$ with any $k>2$ give us no better approximation in the sense of (\ref{anmerkunggleichkont4a}) than the ratio ${n^2}/N\rightarrow 0$.
\end{anmerkung}

Further comparisons with polynomial interpolation, method of least squares on different nodes and polynomial of best approximation you can find in \cite{goertz2018konvergenz}.

\end{subsection}

\end{section}

\newpage
% Literaturliste endgueltig anzeigen
% \nocite{*}
\printbibliography[heading=bibintoc, title=References]
%\bibliography{literatur}

\end{document}